\documentclass[doublespacing]{elsart}

\usepackage{amssymb}
\usepackage[T1]{fontenc}
\usepackage[english,francais]{babel}
\usepackage{ulem}
\newtheorem{theorem}{Theorem}[section]

\newtheorem{proof}[theorem]{Proof}
\newtheorem{e-proposition}[theorem]{Proposition}

\newtheorem{e-definition}[theorem]{Definition\rm}

\def\og{\leavevmode\raise.3ex\hbox{$\scriptscriptstyle\langle\!\langle$~}}
\def\fg{\leavevmode\raise.3ex\hbox{~$\!\scriptscriptstyle\,\rangle\!\rangle$}}

\begin{document}
% place in the next line the header (rubrique) chosen for your article,
% if you know it (you can also have 2, format : Header1/Header2
\centerline{}
\begin{frontmatter}


% \author{Name\thanksref{label2}}
\selectlanguage{english}
\title{About sum rules for Gould-Hopper polynomials}

% The [label1] can be suppressed if there is only one address for all authors

\selectlanguage{english}
\author{Olivier L\'{e}v\^{e}que},
\ead{olivier.leveque@epfl.ch}
\author{Christophe Vignat},
\ead{christophe.vignat@epfl.ch}
\address{L.T.H.I., E.P.F.L., Switzerland}

% If you know the dates of reception, and acceptation you can put them now;
%  idem the name of the person presenting the Note

%\medskip
%\begin{center}
%{\small Received *****; accepted after revision +++++\\
%Presented by £££££}
%\end{center}

\begin{abstract}
\selectlanguage{english}
%In \cite{Graczyk:2004fk}, a composition formula for squares
%of Gould-Hopper and Hermite polynomials is given; it is proved using
%the Mehler formula. We provide here a probabilistic proof that reveals
%the importance of orthogonal invariance of multivariate Gaussian laws.
We show that various identities from \cite{daboul}  and \cite{Graczyk:2004fk} involving Gould-Hopper polynomials can be deduced from the real but also complex orthogonal invariance of multivariate Gaussian distributions. We also deduce from this principle a useful stochastic representation for the inner product of two non-centered Gaussian vectors and two non-centered Gaussian matrices.
%{\it To cite this article: A. Name1, C. R. Acad. Sci. Paris, Ser. I 340 (2005).}

\vskip 0.5\baselineskip

\selectlanguage{francais}
% Text of abstract in French
\noindent{\bf R\'esum\'e} \vskip 0.5\baselineskip \noindent
{\bf Formules de sommation pour les polyn\^{o}mes de Gould et Hopper}
Nous montrons que des identit\'{e}s pour les polyn\^{o}mes de Gould et Hopper issues de \cite{daboul}  et  \cite{Graczyk:2004fk}   peuvent \^{e}tre d\'{e}duites de l'invariance orthogonale r\'{e}elle mais aussi  complexe des lois Gaussiennes multivari\'{e}es. Nous d\'{e}duisons aussi de cette propri\'{e}t\'{e} une repr\'{e}sentation stochastique du produit scalaire de deux vecteurs gaussiens non-centr\'{e}s, ainsi que de deux matrices gaussiennes non-centr\'ees.
%{\it Pour citer cet article: A. Name1, C. R. Acad. Sci.
%Paris, Ser. I 340 (2005).}

\end{abstract}
\end{frontmatter}

% now the Version française abrégée, if it exists
%\selectlanguage{francais}
%\section*{Version fran\c{c}aise abr\'eg\'ee}
% Text of your Version française abrégée here.
% Note you do not need to repeat here equations that you use in the
% main text - for example 'voir (3)' is quite acceptable.

\selectlanguage{english}
% main text
\vspace{-0.5cm}
\section{An identity by Graczyck and Nowak}
\vspace{-0.5cm}
%\label{}
We adopt the multi-index notation: with $n\ge1,$ a multi-index is
denoted as $\underbar{m}=\left(m_{1},\dots,m_{n}\right)\in\mathbb{N}^{n}$
and its length $\vert\underbar{m}\vert=m_{1}+\dots+m_{n}.$ In particular,
the multi-index factorial is $\underbar{m}!=m_{1}!\dots m_{n}!$ and, with
$\underbar{x}\in\mathbb{R}^{n},$ the multi-index power $\underbar{x}^{\underbar{m}}=x_{1}^{m_{1}}\dots x_{n}^{m_{n}}.$
The Gould-Hopper polynomials \cite{Gould:1962uq} are 
\begin{equation}
\label{Gould-Hopper}
g_{m}\left(x,p\right)=\sum_{k=0}^{\left\lfloor m/2\right\rfloor }\frac{m!}{k!\left(m-2k\right)}p^{k}x^{m-2k};\,\, x,p\in\mathbb{R}.
\end{equation}
Their multi-index version reads $
g_{\underbar{m}}\left(\underbar{x},p\right)=\prod_{i=1}^{n}g_{m_{i}}\left(x_{i},p\right).
$
\\
We use the following notations: $\underbar{X}\sim\mathcal{N}(\underbar{m}_{X},\uuline{R_{X}})$
expresses the fact that the vector $\underbar{X}$ is Gaussian with mean
$\underbar{m}_{X}$ and covariance matrix $\uuline{R_{X}};$
the expectation operator is denoted as $\mathbb{E},$ 
the Euclidean norm of vector $\underbar{z}$ as $\vert\underbar{z}\vert$ and $\left(n\right)_{j}$
is  the Pochhammer symbol $\frac{\Gamma\left(n+j\right)}{\Gamma\left(n\right)}.$ Finally, underlined variables denote vectors, random variables are capitalized,   $\underbar{x}^{t}$ denotes the transpose vector of $\underbar{x}$ and $\sim $ means equality in distribution.

Our first result is a stochastic representation of the inner product of two non centered Gaussian vectors.
\begin{theorem}
\label{thm:inner}
If $\underbar{x},$ $\underbar{y} \in \mathbb{R}^{n}$, $p \in \mathbb{R}$ and $\underbar{X},$ $\underbar{Y}$ are independent random vectors $\sim \mathcal{N}(\underbar{0},\uuline{I_{n}})$ then
\begin{equation}
\label{scalar}
\left(\underbar{x}+\sqrt{p}\underbar{N}\right)^{t}\left(\underbar{y}+\sqrt{p}\underbar{M}\right)\sim\left(x+\sqrt{p}N_{1}\right)\left(y+\sqrt{p}M_{1}\right)+pZ_{n-1}N
\end{equation}
where $N,\,N_{1}$ and $M_{1}$ are independent standard Gaussian random variables,
\begin{equation}
\label{polarization}
x=\frac{\vert\underbar{x}+\underbar{y}\vert+\vert\underbar{x}-\underbar{y}\vert}{2},\,\,  \,\, y=\frac{\vert\underbar{x}+\underbar{y}\vert-\vert\underbar{x}-\underbar{y}\vert}{2}
\end{equation}
and  $Z_{n-1}$ is chi-distributed with $n-1$ degrees of freedom and independent of $N,\,N_{1}$ and $M_{1}$.
\end{theorem}

\begin{proof}
The proof is based on the polarization identity, that allows to express this inner product as
\begin{eqnarray*}
\frac{1}{4}\left[\vert\underbar{x}+\sqrt{p}\underbar{N}+\underbar{y}+\sqrt{p}\underbar{M}\vert^{2}-\vert\underbar{x}+\sqrt{p}\underbar{N}-\underbar{y}-\sqrt{p}\underbar{M}\vert^{2}\right]
&   & \\
 \sim \frac{1}{4}\left[\vert\underbar{x}+\underbar{y}+\sqrt{2p}\tilde{\underbar{N}}\vert^{2}-\vert\underbar{x}-\underbar{y}+\sqrt{2p}\tilde{\underbar{M}}\vert^{2}\right]
\end{eqnarray*}
where vectors  $\tilde{\underbar{M}}=\frac{1}{\sqrt{2}}\left(\underbar{N}+\underbar{M}\right)$ and $\tilde{\underbar{N}}=\frac{1}{\sqrt{2}}\left(\underbar{N}-\underbar{M}\right)
%\sim\mathcal{N}\left(0,2I_{n}\right)
$
are again Gaussian and independent as a consequence of the orthogonal invariance. 

The next  step is again a consequence of this invariance: if $\underbar{G}\sim\mathcal{N}(\underbar{0}, \uuline{I_{n}})$
and $\underbar{z}\in\mathbb{R}^{n},$ then the norm of the random vector $\underbar{z}+\underbar{G}$
depends on $\vert\underbar{z}\vert$ only; more precisely
\[
\vert\underbar{z}+\underbar{G}\vert\sim\vert\vert\underbar{z}\vert\underbar{e}_{1}+\underbar{G}\vert=\sqrt{\left(\vert\underbar{z}\vert+G_{1}\right)^{2}+G_{2}^{2}+\dots+G_{n}^{2}}
\]
where\textbf{ $\underbar{e}_{1}$} is the first (or any) column vector
of the identity matrix $\uuline{I_{n}}$. We deduce that 
\[
\left(\underbar{x}+\sqrt{p}\underbar{N}\right)^{t}\left(\underbar{y}+\sqrt{p}\underbar{M}\right) \sim \frac{1}{4}\left[\vert\vert\underbar{x}+\underbar{y}\vert\underbar{e}_{1}+\sqrt{p}\underbar{N}+\sqrt{p}\underbar{M}\vert^{2}-\vert\vert\underbar{x}-\underbar{y}\vert\underbar{e}_{1}+\sqrt{p}\underbar{N}-\sqrt{p}\underbar{M}\vert^{2}\right].
\]
and by the polarization identity, with $X$ and $Y$ as in Theorem \ref{thm:inner}, this expression simplifies to
\[
\left(x\underbar{e}_{1}+\sqrt{p}\underbar{N}\right)^{t}\left(y\underbar{e}_{1}+\sqrt{p}\underbar{M}\right)=
\left(x+\sqrt{p}N_{1}\right)\left(y+\sqrt{p}M_{1}\right)+p\sum_{i=2}^{n}N_{i}M_{i}.
\]

The proof follows from the identity in distribution, with $Z_{n-1} \sim \chi_{n-1}$ independent of $N\sim\mathcal{N}\left(0,1\right)$:
\[
\sum_{i=2}^{n}N_{i}M_{i}\sim \left(\sum_{i=2}^{n}M_{i}^{2}\right)^{\frac{1}{2}}N = Z_{n-1}N
\]
\end{proof}
% We deduce
%\[
%\left(\underbar{x}+\sqrt{p}\underbar{N}\right)^{t}\left(\underbar{y}+\sqrt{p}\underbar{M}\right)\sim\left(X+\sqrt{p}N_{1}\right)\left(Y+\sqrt{p}M_{1}\right)+pZ_{n-1}N
%\]
%where $M_{1},N_{1}$ and $N$ are independent and standard Gaussian.

We deduce an elementary proof of the following \cite[Thm 3]{Graczyk:2004fk}, originally derived using generating functions.
\begin{theorem}
\label{thm:Graczyk}
For all $M\in\mathbb{N},\,\, p\in\mathbb{R}$ and $\underbar{x},\underbar{y}\in\mathbb{R}^{n},$
we have, with $x$ and $y$ as in (\ref{polarization}),
\begin{equation}
\sum_{\vert\underbar{m}\vert=M}\frac{1}{\underbar{m}!}g_{\underbar{m}}\left(\underbar{x},p\right)g_{\underbar{m}}\left(\underbar{y},p\right)=\sum_{j=0}^{\left\lfloor M/2\right\rfloor }\frac{\left(2p\right)^{2j}}{j!\left(M-2j\right)!}\left(\frac{n-1}{2}\right)_{j}g_{M-2j}\left(x,p\right)g_{M-2j}\left(y,p\right)\label{eq:nowak3}
\end{equation}
\end{theorem}
\begin{proof}
We use the moment representation of Gould-Hopper polynomials
\begin{equation}
\label{GHrepresentation}
g_{m}\left(x,p\right)=\mathbb{E}\left(x+\sqrt{2p}N\right)^{m}
\end{equation}
where $N\sim\mathcal{N}\left(0,1\right)$: it is a consequence of definition (\ref{Gould-Hopper}) and of the fact that $\mathbb{E}N^{2k}=\frac{\left(2k\right)!}{2^{k}k!}$ and $\mathbb{E}N^{2k+1}=0$. We then recognize the left-hand
side of (\ref{eq:nowak3}) as the multinomial expansion
\begin{eqnarray*}
\mathbb{E}\sum_{\vert\underbar{m}\vert=M}\frac{1}{\underbar{m}!}\prod_{i=1}^{n}\left(x_{i}+\sqrt{2p}N_{i}\right)^{m_{i}}\left(y_{i}+\sqrt{2p}M_{i}\right)^{m_{i}} 
&  &\\
 =\frac{1}{M!}\mathbb{E}\left(\sum_{i=1}^{n}\left(x_{i}+\sqrt{2p}N_{i}\right)\left(y_{i}+\sqrt{2p}M_{i}\right)\right)^{M}\label{eq:powerM}
\end{eqnarray*}
where $\underbar{M}$ and $\underbar{N}\sim\mathcal{N}\left(0,\uuline{I_{n}}\right)$ are independent. The
 sum is identified as the inner product 
\[
(\underbar{x}+\sqrt{2p}\underbar{N})^{t}(\underbar{y}+\sqrt{2p}\underbar{M}).
\]
Using the multinomial expansion of the $M-th$ power of the stochastic equivalent (\ref{scalar}) of this expression
and taking expectations with  $\mathbb{E}Z_{n-1}^{2j}=2^{j}\left(\frac{n-1}{2}\right)_{j},$
we obtain the desired result.
\end{proof}
\vspace{-0.5cm}
\section{Two identities by Daboul and Mizrahi}
\vspace{-0.5cm}
%\subsection{Introduction}
The group of complex rotations $\mathcal{O}_{\mathbb{C}}\left(n\right)$ is the set of $n \times  n$ complex matrices $\uuline{O}$ such that $\uuline{O} \, \uuline{O}^{t}=\uuline{O}^{t} \, \uuline{O}=\uuline{I_{n}}.$ The \textbf{real} inner product $\underbar{x}^{t} \, \underbar{y}=\sum_{i=1}^{n}x_{i}y_{i}, \,\, \underbar{x},\,\,\underbar{y} \in \mathbb{C}^{n},$ is preserved under the action of this group.
In \cite{daboul}, sum rules for Gould-Hopper polynomials are proved using the covariant transformation property under $\mathcal{O}_{\mathbb{C}}\left(n\right)$ of the raising operators associated with these polynomials. We show here that these sum rules can be equivalently deduced from the following complex rotational invariance principle.
% for standard Gaussian vectors, and hence for Gould-Hopper polynomials, that we state as % follows.

\begin{theorem}
\label{prop:rotational}
For any $\uuline{O} \in \mathcal{O}_{\mathbb{C}}\left(n\right)$ and  $\underbar{N}\sim \mathcal{N}(0,\uuline{I_{n}})$, the Gould-Hopper polynomial of degree $m$ reads
\[
g_{m}\left(x,p\right)=\mathbb{E}\left(x+\sqrt{2p}\left(\uuline{O} \, \underbar{N}\right)_{j}\right)^{m},\,\,\forall j\in\left[1,n\right].
\]
\end{theorem}

%This identity can be proved by computing the generating function
%\[
%\sum_{k=0}^{+\infty}2^{k}\mathbb{E}\left(x+\imath\left(M\underbar{z}\right)_{j}\right)^{k}\frac{t^{k}}{k!}=\exp\left(2tx\right)\mathbb{E}\exp\left(2\imath t\left(M\underbar{z}\right)_{j}\right)
%\]
%with
%\begin{equation}
%\mathbb{E}\exp\left(2\imath t\left(M\underbar{z}\right)_{j}\right) =\prod_{k=1}^{n}\mathbb{E}\exp\left(2\imath tM_{jk}z_{k}\right)
% =  \prod_{k=1}^{n}\exp\left(-t^{2}M_{jk}^{2}\right)=\exp\left(-t^{2}\sum_{k=1}^{n}M_{jk}^{2}\right).
%\end{equation}
\begin{proof}
We only need to check that $\left(\uuline{O} \, \underbar{N}\right)_{j} \sim \mathcal{N}\left(0,1\right)$; but, using the fact that the identity
$\mathbb{E}\exp\left(2\imath tN_{j}\right)=\exp\left(-t^{2}\right)$
holds for any \textbf{complex} value of $t$, the characteristic function of this random variable reads
\[
\mathbb{E}\exp\left(\imath t\left(\uuline{O} \, \underbar{N}\right)_{j}\right) 
%=\prod_{k=1}^{n}\mathbb{E}\exp\left(2\imath tM_{jk}z_{k}\right)
% =  \prod_{k=1}^{n}\exp\left(-t^{2}M_{jk}^{2}\right)
=\exp\left(-t^{2}\sum_{k=1}^{n}O_{jk}^{2}\right)=\exp\left(-t^{2}\right)
\]
by the complex orthogonality condition.
%we are left with $\exp\left(-t^{2}\right)$,
%so that the computed generating function coincides with that of the
%Hermite polynomials. 
%We note that the key point in this proof is
\end{proof}
Using this probabilistic invariance, we derive a simple proof of the summation theorem \cite[Prop. 1]{daboul} 
\begin{theorem}
For any $\uuline{O}\in \mathcal{O}_{\mathbb{C}}\left(n\right),$
\[
g_{m}\left(\left(\uuline{O} \, \underbar{x}\right)_{i},p\right)=\sum_{\vert \underbar{m}\vert=m}{m \choose m_{1},\dots,m_{n}}\prod_{j=1}^{n} O_{ij}^{m_{j}}g_{m_{j}}\left(x_{j},p\right)
\] 
\end{theorem}
\begin{proof}
The proof is based on the complex rotational invariance (Theorem \ref{prop:rotational}) as follows
\begin{eqnarray*}
g_{m}\left(\left(\uuline{O} \, \underbar{x}\right)_{i},p\right) =\mathbb{E}\left(\left(\uuline{O} \, \underbar{x}\right)_{i}+ \sqrt{2p} \left(\uuline{O} \, \underbar{N}\right)_{i}\right)^{k} \sim \mathbb{E}\left(\sum_{j=1}^{n}O_{ij}\left(x_{j}+\sqrt{2p}N_{j}\right)\right)^{m}.
\end{eqnarray*}
The desired result is obtained by expanding this power using the multinomial formula
as
\[
\hspace{-0.8cm}
\mathbb{E}\sum_{\vert \underbar{m}\vert=m}{m \choose m_{1},\dots,m_{n}} \prod_{j=1}^{n} O_{ij}^{m_{j}}\left(x_{j}+\sqrt{2p}N_{j}\right)^{m_{j}}=\sum_{\vert \underbar{m}\vert=m}{m \choose m_{1},\dots,m_{n}}\prod_{j=1}^{n}O_{ij}^{m_{j}}g_{m_{j}}\left(x_{j},p\right)
.\]
\end{proof}

%\subsection{The factorization sum rule}
We now prove equally easily the following  generalized factorization sum rule \cite[Prop.3]{daboul}
\begin{theorem} With $p\in\mathbb{R},\,\,c,s,x,y\in\mathbb{C}$ and $c^{2}+s^{2}=1$,
\[
g_{m_{1}}\left(cx-sy,p\right)g_{m_{1}}\left(sx+cy,p\right)=\sum_{r=0}^{m_{1}+m_{2}} C_{m_{1},m_{2},r}\left(c,s\right)g_{r}\left(x,p\right)g_{m_{1}+m_{2}-r}\left(y,p\right)
\]
where coefficients $C_{m_{1},m_{2},r}\left(c,s\right)$ are given by
\begin{equation}
C_{m_{1},m_{2},r}\left(c,s\right)=\sum_{l=0}^{m_{2}\wedge r}{m_{1} \choose r-l}{m_{2} \choose l}\left(-1\right)^{m_{1}-r+l}c^{m_{2}+r-2l}s^{m_{1}-r+2l}.\label{eq:coeffC}
\end{equation}
\end{theorem}
\begin{proof}
By the complex orthogonal invariance,  $\left[N_{1},N_{2}\right] \sim \mathcal{N}\left(0,I_{2}\right)\sim \left[cN_{1}-sN_{2},sN_{1}+cN_{2}\right]$ so that
\begin{eqnarray*}
& &\hspace{-0.5cm}g_{m_{1}}\left(cx-sy,p\right)g_{m_{2}}\left(sx+cy,p\right)\\
& &=
\mathbb{E}\left(cx-sy+\sqrt{2p}\left(cN_{1}-sN_{2}\right)\right)^{m_{1}}\mathbb{E}\left(sx+cy+\sqrt{2p}\left(sN_{1}+cN_{2}\right)\right)^{m_{2}}\\
& &=\mathbb{E}\sum_{k=0}^{m_{1}}\sum_{l=0}^{m_{2}}{m_{1} \choose k}{m_{2} \choose l} \left(-1\right)^{m_{1}-k} c^{m_{2}+k-l} s^{m_{1}+l-k}\left(x+\sqrt{2p} N_{1}\right)^{l+k}\left(y+\sqrt{2p} N_{2}\right)^{m-l+n-k}\\
& &=\sum_{r=0}^{m_{1}+m_{1}}C_{m_{1},m_{2},r}\left(c,s\right)g_{r}\left(x,p\right)g_{m_{1}+m_{2}-r}\left(y,p\right)
\end{eqnarray*}
with coefficients $C_{m_{1},m_{2},r}\left(c,s\right)$ as in (\ref{eq:coeffC}),
which coincides with \cite[eq. 19]{daboul}. 
\end{proof}
%\subsection{A sum rule based on rotational invariance}
%Finally, still using complex rotational invariance, we prove easily the following \cite[Prop. 5]{daboul}
%\begin{theorem}
%With $w_{i}=\left(\uuline{O}x\right)_{i},$ the following sum rule holds
%\begin{equation}
%\label{Daboul3}
%\sum_{\vert \underbar{m}\vert=m}\frac{m!}{m_{1}!\dots m_{n}!}\prod_{j=1}^{n}g_{2m_{j}}\left(w_{j},p\right)=\sum_{\vert \underbar{m}\vert=m}\frac{m!}{m_{1}!\dots m_{n}!}\prod_{j=1}^{n}g_{2m_{j}}\left(x_{j},p\right)
%\end{equation}
%\end{theorem}
%\begin{proof}
%Using representation(\ref{GHrepresentation}), the left-hand side of (\ref{Daboul3}) is
%\begin{eqnarray*}
%\mathbb{E}\sum_{\vert \underbar{m}\vert=m}\frac{m!}{\underbar{m}!}\prod_{j=1}^{n}\left(w_{j}+\sqrt{2p} N_{j}\right)^{2m_{j}}  \hspace{-0.3cm}=  \mathbb{E}\left(\sum_{j=1}^{n}\left(w_{j}+\sqrt{2p} N_{j}\right)^{2}\right)^{m} \hspace{-0.3cm}=  \mathbb{E} \left(\sum_{j=1}^{n}\left(\left(\uuline{O}x\right)_{j}+\sqrt{2p}\left(O\underbar{N}\right)_{j}\right)^{2}\right)^{m}
%\end{eqnarray*}
%which coincides with the right-hand side of (\ref{Daboul3}).
%\end{proof}

We note that all former sum rules apply to the Hermite polynomials since
$H_{n}\left( x \right) = g_{n}\left(2x, \sqrt{-2} \right).$
Moreover, theorem \ref{thm:inner} extends to the matrix case as follows, with $\Vert \uuline{X} \Vert_{F}$ denoting the Frobenius norm:
\begin{theorem}
If $\uuline{x},\uuline{y} \in \mathbb{R}^{m\times n}$  and $\uuline{M},Ê\uuline{N} \in \mathbb{R}^{m\times n}$ are independent matrices with i.i.d. standard complex Gaussian entries then, with  $M_{1}, N_{1}, N\sim \mathcal{N}\left(0,1\right)$ and $Z_{mn-1}\sim \chi_{mn-1}$ independent,
\[
\mathtt{tr}(\uuline{x}+\uuline{N})^{t}(\uuline{y}+\uuline{M}) \sim \left(x+N_{1}\right)\left(y+M_{1}\right) + Z_{mn-1}N
\]
with $x=1/2(\Vert \uuline{x}+\uuline{y} \Vert_{F}^{2} + \Vert \uuline{x}-\uuline{y} \Vert_{F}^{2}),\, y=1/2(\Vert \uuline{x}+\uuline{y} \Vert_{F}^{2} - \Vert \uuline{x}-\uuline{y} \Vert_{F}^{2}).$ 
\end{theorem}
\begin{proof}
Denoting $\underbar{x}_{i}$ the $i-$th column vector of $\uuline{x}$ and $\underbar{u},\underbar{v}$ the vectors with components $u_{i},v_{i}= 1/2(\Vert \underbar{x}_{i}+\underbar{y}_{i} \Vert^{2} \pm \Vert \underbar{x}_{i}-\underbar{y}_{i} \Vert),$ we deduce that $1/2\left(\Vert \underbar{u}+\underbar{v} \Vert^{2} \pm \Vert \underbar{u}-\underbar{v} \Vert\right)=x,y$ and by theorem \ref{thm:inner}
\[
\mathtt{tr}(\uuline{x}+\uuline{N})^{t}(\uuline{y}+\uuline{M}) = \sum_{i=1}^{m} 
(\underbar{x}_{i}+\underbar{N}_{i})^{t}(\underbar{y}_{i}+\underbar{M}_{i})
\sim \sum_{i=1}^{m} 
\left(u_{i}+N_{i}\right)\left(v_{i}+M_{i}\right) + \chi_{n-1}^{\left(i\right)}G_{i}.
\] 
Using again theorem \ref{thm:inner}, the first sum is distributed as $\left(x+N_{1}\right)\left(y+M_{1}\right) + \chi_{m-1}G$ and since
$\chi_{m-1}G_{2} + \sum_{i=1}^{m} \chi_{n-1}^{\left(i\right)}G_{i} \sim \chi_{nm-1}N,$ the result follows.
\end{proof}
%The proof follows the steps of that of theorem \ref{thm:inner}; we skip the details.
% The Acknowledgements are an un-numbered section
%\section*{Acknowledgements}
%The author thanks the LTHI laboratory for its invitation.

\end{document}